\documentclass[reqno]{amsart}
\usepackage[titletoc,title]{appendix}
\usepackage[utf8]{inputenc}
\usepackage[english]{babel}
\usepackage{graphicx,amsmath,amsthm,amsfonts,amssymb,amscd,amsmath,caption,color,epsf,epstopdf,float,latexsym,mathrsfs,soul,wasysym}
\usepackage[hidelinks]{hyperref}
\usepackage[section]{placeins}

\def\deg{{\rm deg}\,}

\newtheorem{example}{Example}
\newtheorem{lemma}{Lemma}
\newtheorem{proposition}{Proposition}

\newtheorem{thm}{Theorem}
\let\xmpl\example
\newcommand \Vor {\rm{\text{Vor}}}
\newcommand {\D} {\mathcal D}
\renewcommand{\example}{\xmpl\normalfont}

\newcommand*\diff{\mathop{}\!\mathrm{d}}

\title[The asymptotic zero-counting measure of iterated derivaties]{The asymptotic zero-counting measure of iterated derivaties of a class of meromorphic functions}
\author[C.~H\"agg]{Christian H\"agg}
\address{Department of Mathematics, Stockholm University, SE-106 91 Stockholm, Sweden}
\email{hagg@math.su.se}

\begin{document}

\begin{abstract}
\noindent We give an explicit formula for the logarithmic potential of the asymptotic zero-counting measure of the sequence $\left\{\frac{\mathrm{d}^n}{\mathrm{d}z^n}\left(R(z)\exp{T(z)}\right)\right\}$. Here, $R(z)$ is a rational function with at least two poles, all of which are distinct, and $T(z)$ is a polynomial. This is an extension of a recent measure-theoretic refinement of P\'olya's Shire theorem for rational functions.
\end{abstract}

\maketitle

\section{Introduction}
Consider a meromorphic function $f$, and let $S$ denote its set of poles. P\'olya proved in 1922 that the zeros of the iterated derivatives $f',\,f'',\,f''',\dotsc$ of such a function asymptotically accumulate along the boundaries of the Voronoi diagram associated with $S$. This classical result is called P\'olya's Shire theorem (see \cite{Ha, Po1}). In a recent paper by Rikard Bögvad and this author (see \cite{BoHa}), a measure-theoretic refinement of P\'olya's Shire theorem was given for the special case that $f = P/Q$, where $P$ and $Q$ are polynomials with $\gcd(P,Q) = 1$, and $P\not\equiv 0$.

In this paper, we generalize the main result of the aforementioned paper (see Theorem 1 of \cite{BoHa}) to the situation when $f = (P/Q)e^T$, where $P$ and $Q$ are defined as previously, and $T$ is a nonconstant polynomial. Furthermore, we assume that $Q$ is monic and has at least two zeros, all of which are distinct. Under these conditions, it follows from Hadamard's factorization theorem (see \cite{Ti}) that the class of such functions is equivalent to the class of meromorphic functions that are quotients of two entire functions of finite order, each with a finite number of zeros. For convenience, we denote $p := \deg{P},\,q := \deg{Q}$ and $t := \deg{T}$ throughout this paper, and additionally set $P = \sum_{k=0}^{p} b_kz^k,\,Q = \sum_{k=0}^{q} c_kz^k$ and $T = \sum_{k=0}^{t} d_kz^k$.

Before we state the main result of this paper in Theorem \ref{thm:LogPotentialMeasure} below, we remind the reader that if $\widetilde{P}(z)$ is a polynomial of degree $d\ge 1$, then its \textit{zero-counting measure} $\mu$ is a probability measure that assigns mass $1/d$ to each zero of $\widetilde{P}(z)$, accounting for multiplicity (see \cite{BeRu}).

\begin{thm}\label{thm:LogPotentialMeasure}
Let $f := (P/Q)e^T$, where $P,\,Q$ and $T$ are polynomials with $\gcd(P,Q) = 1,\,P\not\equiv 0,\,\deg{Q}\ge 2$ and $\deg{T}\ge 1$. Furthermore, assume that $Q$ is monic, and that all of its zeros $z_1,\dotsc,z_q$ are distinct. Then
\\(i) the zero-counting measures $\mu_n$ of the sequence $\left\{f^{(n)}\right\}_{n=1}^{\infty}$ converge to a measure $\mu_{_S}$ with mass $(q-1)/(q-1+t)$.
\\(ii) The logarithmic potentials $\mathcal{L}_{\mu_n}(z)$ of $\mu_n$ diverge as $n\to\infty$. 
\\(iii) The shifted logarithmic potentials $\widetilde{\mathcal{L}}_{\mu_n}(z) := \mathcal{L}_{\mu_n}(z) - (\log{n!})/(n(q+t-1)+p)$ of $\mu_n$ converge in $L_{loc}^1$ to the distribution $\Psi(z)$, where
\begin{equation}\label{eq:logPotentialOfTheAsymptoticRootMeasure}
\Psi(z) = \frac{1}{q+t-1}\left(\max_{i=1,\dotsc,q}\left\{\log{\left\vert z-z_i\right\vert^{-1}}\right\} + \log{\vert Q\vert} - \log{\left(\left\vert d_t\right\vert t\right)}\right).
\end{equation}
\\(iv) The measure $\mu_{_S}$ is given by $\frac{1}{2\pi}\Delta\Psi(z).$
\end{thm}

In the terminology of Theorem \ref{thm:LogPotentialMeasure}, it is intuitive to refer to $\Psi(z)$ as the shifted logarithmic potential of $\mu_{_S}$. Additionally, note that the formula used to reconstruct the measure $\mu_{_S}$ in (iv) is identical to the formula used in the reconstruction of a measure from its associated logarithmic potential (see \cite{BeRu}). Furthermore, note that if $t = 0$, it follows from Theorem 1 of \cite{BoHa} that the logarithmic potential of the asymptotic zero-counting measure $\mu$ of the sequence $\left\{(P/Q)^{(n)}\right\}_{n=1}^{\infty}$ is given by
$$\mathcal{L}_\mu(z) = \frac{1}{q-1}\left(\max_{i=1,\dotsc,q}\left\{\log{\left\vert z-z_i\right\vert^{-1}}\right\} + \log{\vert Q\vert}\right).$$
Thus, there are strong similarities with Theorem \ref{thm:LogPotentialMeasure} above. An illustration of Theorem \ref{thm:LogPotentialMeasure} is given in Figure \ref{fig:VoronoiZeros}.

\begin{figure}[htp]
\begin{center}
\includegraphics[width=.48\textwidth]{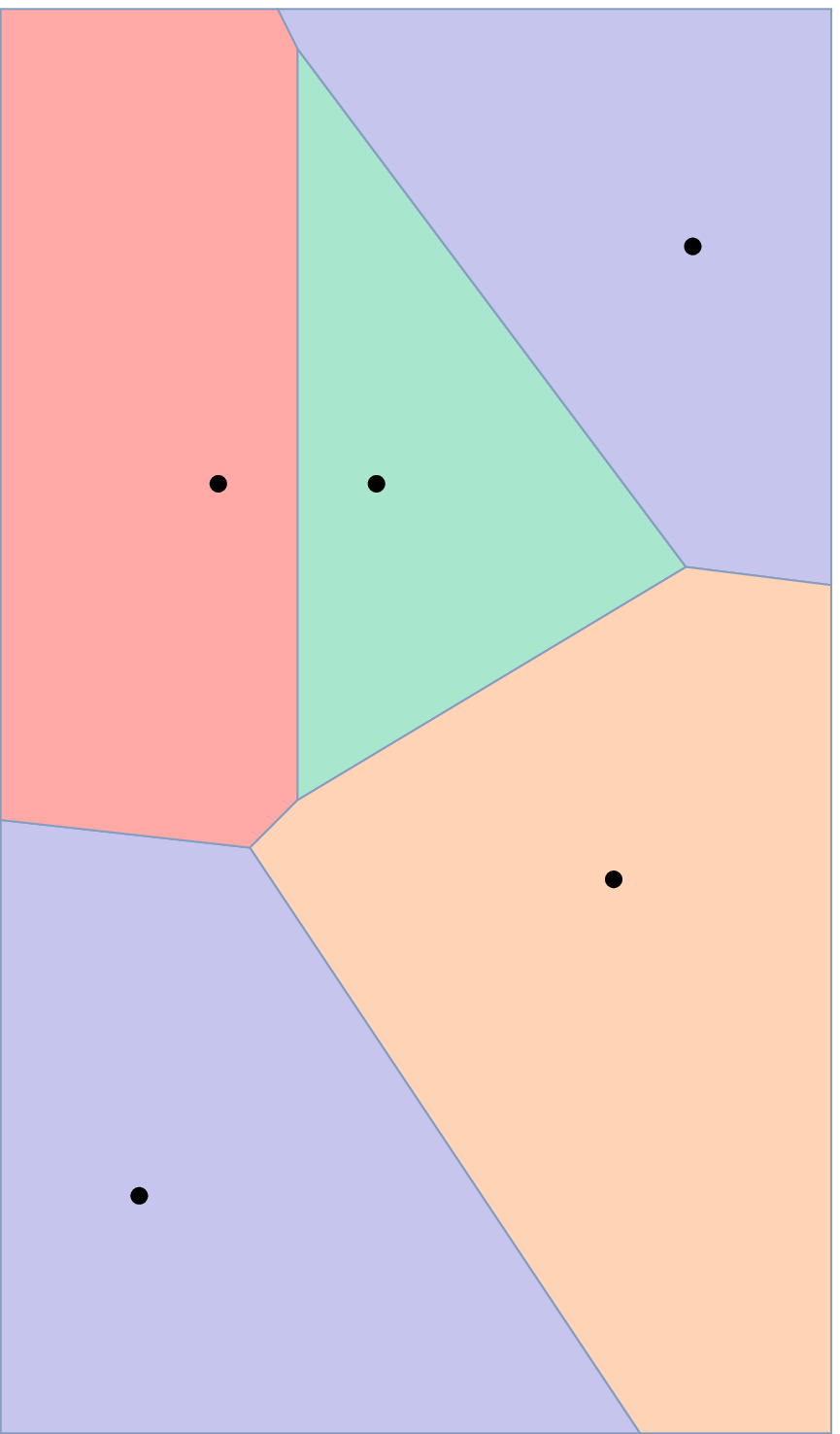}\hfill
\includegraphics[width=.48\textwidth]{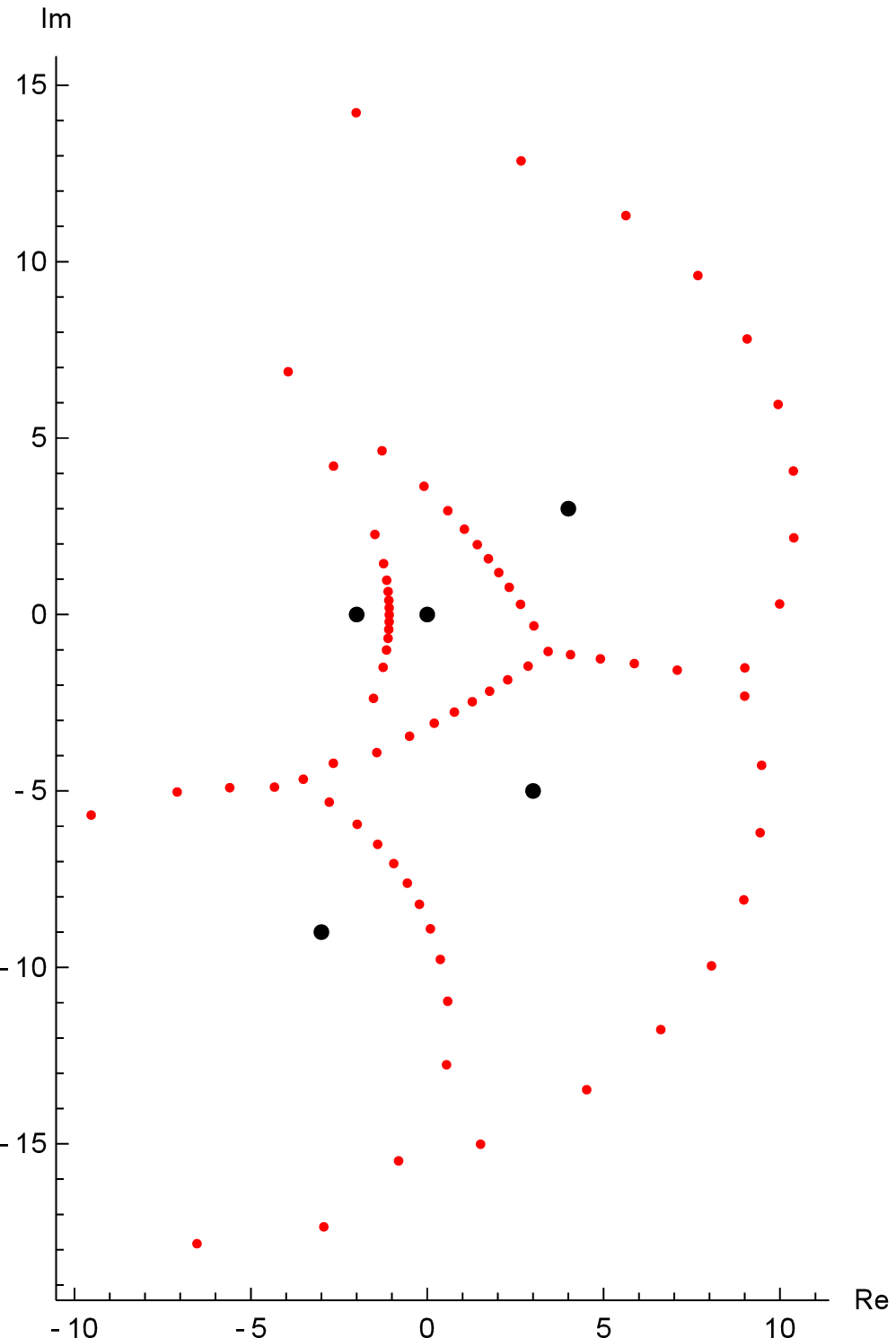}
\end{center}
\caption{The Voronoi diagram generated by the zeros of the polynomial $Q = z(z+2)(z-4-3i)(z-3+5i)(z+3+9i)$ (left), and the $75$ zeros of $\left((1/Q)e^T\right)^{(15)}$ (small dots), where $T = z+1$ (right). Note that $56/75\approx 4/5 = (q-1)/(q+t-1)$ of the zeros approximately appear to be supported on the Voronoi diagram, in accordance with Theorem \ref{thm:LogPotentialMeasure}.}
\label{fig:VoronoiZeros}
\end{figure}

The connection between the probability measures $\mu_n$ and the measure $\mu_{_S}$ with mass $(q-1)/(q+t-1) < 1$ (which we will detail later in Proposition \ref{prop:notAProbabilityMeasure}) may seem surprising, as it implies that a mass of $t/(q+t-1)$ disappears as $n\to\infty$. This mass discrepancy appears to arise due to the ``bubbles" in Figure \ref{fig:VoronoiZeros}, whose structure does not appear to converge on the Voronoi diagram of $Z(Q)$ (compare this to Figure 1 in \cite{BoHa}, where $t=0$, and no such structures seem to arise). Numerical experiments indicate that these ``bubbles" expand toward $\infty$ asymptotically.

The author is indebted to Rikard Bögvad for discussions, ideas, corrections and comments, and to Boris Shapiro for additional corrections and suggestions.

\section{Voronoi diagrams}\label{sec:VoronoiDiagrams}

Consider a set of $q$ distinct points $S = \{z_1,\dotsc,z_q\}\subset\mathbb{C}$. The Voronoi diagram associated with $S$, denoted by $\Vor_{_S}$, is a partitioning of $\mathbb{C}$ into $q$ distinct cells $V_1,\dotsc,V_q$, where any interior point $\alpha_i$ in $V_i$ is closest to $z_i$ of all points in $S$. The boundary between two adjacent cells $V_i$ and $V_j$ consists of a segment of the line $\vert z-z_i\vert = \vert z-z_j\vert$.

Based on the aforementioned definition of Voronoi diagrams, it is natural to stratify the complex plane using the function
$$\Phi(z) := \min_{i=1,\dotsc,q}\{\vert z-z_i\vert\}.$$
Thus, the (closed) cell $V_i$ that contains the point $z_i\in S$ is equal to the set
$$V_i = \{z:\Phi(z) = \vert z-z_i\vert\}.$$
Similarly, the boundary $V_{ij}$ between two cells $V_i$ and $V_j$ is given by
$$V_{ij} = \{z:\Phi(z) = \vert z-z_i\vert = \vert z-z_j\vert\}.$$
Together, these boundaries form the $1-skeleton$ of $\Vor_{_S}$, which we denote by $\Vor_{_S}^{_B}$ (where $B$ means boundary). Finally, the vertices of $\Vor_{_S}$ are points $z$ such that at least three distances $\vert z-z_i\vert,\,i=1,\dotsc,q$, coincide with $\Phi(z)$.

\section{Uniform convergence of the shifted logarithmic potentials}

We return to the function $f = (P/Q)e^T$, defined as in Theorem \ref{thm:LogPotentialMeasure}. By P\'olya's Shire theorem, the zeros of the iterated derivatives $f',\,f'',\,f''',\dotsc$ of $f$ tend to accumulate along $\Vor_{_S}^{_B}$, where $S = \{z_1,\dotsc,z_q\}$ is the set of zeros of $Q$. Simple computations show that
\begin{equation} \label{eq:explicitFormula}
f^{(n)} = \frac{P_n}{Q^{n+1}}e^T,
\end{equation}

where $P_n$ is a polynomial such that $P_n$ and $Q$ are relatively prime. Clearly, the zeros of $f^{(n)}$ are the zeros of the polynomial $P_n$, so it is of interest to investigate the structure of $P_n$. It follows trivially from \eqref{eq:explicitFormula} that
\begin{equation}\label{eq:recursion002}
P_n = (QT'-nQ')P_{n-1} + QP'_{n-1},\,n\ge 1,
\end{equation}
where $P_0 := P$.

\begin{example}\label{example:simpleNumeratorPolynomials}
If $f=e^z/(z(z-1))$, it follows that $P_0 = 1,\,P_1=z^2-3z+1,\,P_2=z^4-6z^3+13z^2-8z+2$, and $P_3=z^6-9z^5+36z^4-73z^3+63z^2-30z+6$.
\end{example}

To proceed, we make use of the assumptions that $t = \deg{T}\ge 1$ and $P\not\equiv 0$ in Theorem \ref{thm:LogPotentialMeasure}. In this situation, we see from equation \eqref{eq:recursion002} that the $QT'P_{n-1}$ term dominates the degree of $P_n$. Thus, it follows that
\begin{equation} \label{eq:degRn}
\deg{P_n} = \deg{(QT'P_{n-1})} = n(q + t - 1) + p.
\end{equation}

In addition to the degree of $P_n$, we will soon make use of the coefficient $A_n$ of the highest-power term of $P_n$. To determine it explicitly, let $\alpha_1,\dotsc,\alpha_{_{\deg{P_n}}}$ be the zeros of $P_n$, and let $P_n = A_n \prod_{k=1}^{\deg{P_n}} (z-\alpha_k)$. Now note that $A_0 = b_p = (d_t\,t)^{0}\,b_p$. Since $A_n$ depends only on the $QT'P_{n-1}$ term in \eqref{eq:recursion002} when $t\ge 1,\,A_{n} = d_t\,t\,A_{n-1},$ for all $n\ge 1$, and thus,
\begin{equation}\label{eq:AnFormula}
A_n = (d_t\,t)^{n}\cdot b_p,\quad n\ge 0,\quad t\ge 1.
\end{equation}

According to P\'olya's Shire theorem, $\left\vert\frac{f^{(n)}}{n!}\right\vert^{\frac{1}{n}}$ converges pointwise a.e. in any open Voronoi cell to $\max\left\{\frac{1}{\vert z-z_i\vert},\ i=1,\dotsc,q\right\}$. To make use of this, we see from equation \eqref{eq:explicitFormula} that
\begin{equation}\label{eq:logFormula}
\begin{split}
\log{\left(\left\vert \frac{f^{(n)}}{n!}\right\vert^{\frac{1}{n}}\right)} &= \frac{\log{\left\vert A_n\right\vert}}{n} + \frac{\deg{P_n}}{n}\cdot\frac{\log{\left\vert\prod_{k=1}^{\deg{P_n}} (z-\alpha_k)\right\vert}}{\deg{P_n}} + \\
&+ \frac{\log{\left\vert e^T\right\vert}}{n} - \frac{\log{\left\vert n!\right\vert}}{n} - \frac{(n+1)\log{\left\vert Q\right\vert}}{n}.
\end{split}
\end{equation}
Note that the term $\frac{\log{\left\vert \prod_{k=1}^{\deg{P_n}} (z-\alpha_k)\right\vert}}{\deg{P_n}}$ is the logarithmic potential $\mathcal{L}_{\mu_n}(z)$ of the zero-counting measure $\mu_n$ of $P_n / A_n$. Passing to the limit in $n$ in equation \eqref{eq:logFormula} and making use of \eqref{eq:degRn} and \eqref{eq:AnFormula}, we see that
\begin{equation}\label{eq:limit001}
\lim_{n\to\infty}\frac{\log{\left\vert A_n\right\vert}}{n} = \log{\left(\left\vert d_t\right\vert t\right)},
\end{equation}
\begin{equation}\label{eq:limit002}
\lim_{n\to\infty}\frac{\deg{P_n}}{n} = q+t-1,
\end{equation}
\begin{equation}\label{eq:limit003}
\lim_{n\to\infty}\frac{\log{\left\vert e^T\right\vert}}{n} = 0,
\end{equation}
\begin{equation}\label{eq:limit004}
\lim_{n\to\infty}-\frac{\log{n!}}{n} = -\infty,
\end{equation}
and
\begin{equation}\label{eq:limit005}
\lim_{n\to\infty}-\frac{(n+1)\log{\left\vert Q\right\vert}}{n} = -\log{\left\vert Q\right\vert}.
\end{equation}
Thus, since the left-hand side of \eqref{eq:logFormula} converges to
\begin{equation} \label{eq:logMax}
\log{\left(\max_{i=1,\dotsc,q}\left\{\frac{1}{\vert z-z_i\vert}\right\}\right)} = \max_{i=1,\dotsc,q}\left\{\log{\left\vert z-z_i\right\vert^{-1}}\right\}
\end{equation}
inside open Voronoi cells, which is finite outside of $S$, it follows from \eqref{eq:limit001}-\eqref{eq:logMax} that $\lim_{n\to\infty} \mathcal{L}_{\mu_n}(z) = \infty$. This proves part (ii) of Theorem \ref{thm:LogPotentialMeasure}.

Although the logarithmic potential of the asymptotic zero-counting measure $\mu_{_S}$ diverges as $n\to\infty$, the shifted logarithmic potential $\widetilde{\mathcal{L}}_{\mu_n}(z)$ of $\mu_n$ (defined as in part (iii) of Theorem \ref{thm:LogPotentialMeasure}) can be used to rewrite equation \eqref{eq:logFormula} as
\begin{equation}\label{eq:shiftedLogPotential}
\widetilde{\mathcal{L}}_{\mu_n}(z) = \frac{n}{n(q+t-1)+p}\left(\log{\left(\left\vert \frac{f^{(n)}}{n!}\right\vert^{\frac{1}{n}}\right)} + \frac{(n+1)\log{\vert Q\vert}}{n} - \frac{\log{\vert A_n\vert}}{n} - \frac{\log{\vert e^T\vert}}{n}\right),
\end{equation}
or, as we will find use for later, by using the expression for $f^{(n)}$ in \eqref{eq:explicitFormula},
\begin{equation}\label{eq:shiftedLogPotential002}
\widetilde{\mathcal{L}}_{\mu_n}(z) = \frac{1}{n(q+t-1)+p}\left(\log{\left\vert \frac{P_n}{A_n}\right\vert} - \log{n!}\right).
\end{equation}

By letting $n\to\infty$ in \eqref{eq:shiftedLogPotential}, we obtain the equation \eqref{eq:logPotentialOfTheAsymptoticRootMeasure}, where the right-hand side has converged pointwise (in any open Voronoi cell $V_i^o$) to a \textit{continuous subharmonic function} defined in the whole complex plane, as we will see in Lemma \ref{lemma:ContinuousSubharmonic} in the next section. More generally, we have the following proposition, the proof of which is analogous to that of Proposition 4.5 in \cite{BoHa}, and is omitted for brevity.

\begin{proposition}\label{prop:logPotentialConvergence}
Let $\mathcal{L}_{\mu_n}(z) = \frac{\log\vert P_n\vert - \log\vert A_n\vert}{\deg{P_n}}$ be the logarithmic potential of the zero-counting measure $\mu_n$ of $P_n / A_n$. Furthermore, let $\widetilde{\mathcal{L}}_{\mu_n}(z) = \mathcal{L}_{\mu_n}(z) - \frac{\log{n!}}{n(q+t-1)+p}$. Then for any $z$ in the interior of the Voronoi cell $V_i^o$, we have pointwise convergence
\begin{equation}\label{eq:pointwiseConvergence}
\lim_{n\to\infty} \widetilde{\mathcal{L}}_{\mu_n}(z) = \frac{1}{q+t-1}\left(\max_{i=1,\dotsc,q}\left\{\log{\left\vert z-z_i\right\vert^{-1}}\right\} + \log{\vert Q\vert} - \log{\left(\left\vert d_t\right\vert t\right)}\right) =: \Psi(z).
\end{equation}
The convergence is uniform on compact subsets of $V_i^o$.
\end{proposition}

\section{The subharmonic function $\Psi(z)$}

The two results in this section describe properties of the asymptotic zero-counting measure of $P_n/A_n$. Their proofs are analogous to those of Lemma 2.1 and Proposition 2.2 in \cite{BoHa}, respectively.

\begin{lemma}\label{lemma:ContinuousSubharmonic}
The function $\Psi(z)$, defined in $\mathbb{C}$, is a continuous subharmonic function, and is harmonic in the interior of any cell $V_i$.
\end{lemma}

Since $\Psi(z)$ is subharmonic, $\Delta\Psi(z) = 4\frac{\partial^2\Psi(z)}{\partial\bar{z}\partial z}$ is a positive measure with support on $\Vor_{_S}^{_B}$.

The following proposition provides the definition and some properties of what will turn out to be the asymptotic zero-counting measure.

\begin{proposition}
\label{prop:notAProbabilityMeasure}
For each pair $i,j$, define a measure with support on the line $l_{ij}:\ \vert z-z_i\vert = \vert z-z_j\vert$ as
$$
\delta_{ij}=\frac{1}{4(q+t-1)}\frac{\vert  z_i-z_j \vert}{\vert (z-z_i)(z-z_j)\vert}\diff s,
$$
where $\mathrm{d}s$ is the Euclidean length measure in the complex plane. Then
\begin{enumerate}
\item  $\frac{\partial^2 \Psi}{\partial \bar z\partial z}$ is the sum of all $\delta_{ij}$, each restricted to $V_{ij}$.

\item $\mu_{_S}:=\frac{2}{\pi}\frac{\partial^2 \Psi}{\partial \bar z\partial z}$ has mass $(q-1)/(q+t-1)$.
\end{enumerate}
\end{proposition}

\section{Proof of the main theorem}

Uniform convergence a.e. as in Proposition \ref{prop:logPotentialConvergence} does not by itself imply convergence of the logarithmic potentials in $L_{loc}^1$, though it tells us that there is only one possible limit, since a function in $L_{loc}^1$ is determined by its behavior a.e. We will prove the $L_{loc}^1$-convergence directly, with the main difficulty being the unboundedness of the zeros of $P_n$ as $n\to\infty$. To deal with this problem, we give rough bounds of the growth of the zeros of $P_n$ in Lemma \ref{lemma:zeroBounds} below.
\subsection{Growth of zeros} Consider a fixed meromorphic function $f(z) := (P/Q)e^T$ as in Theorem \ref{thm:LogPotentialMeasure}. Lemma \ref{lemma:scalingTranslationInvariance} below shows that if the statement of the theorem holds for $f(z)$, it also holds for $\widehat{f}(z) := f(\tau z+a),\,\tau\in\mathbb{R_+},\,a\in\mathbb{C}$, i.e. the statement of the theorem is invariant under scaling and translation. For convenience, let $\widehat{\mu}_n$ be the zero-counting measure of $\widehat{f}^{(n)}$ (or, technically, of the polynomial $\prod_k(z-\widehat{\alpha}_k)$, where the product is taken over all zeros $\widehat{\alpha}_1,\,\widehat{\alpha}_2,\,\dotsc$ of $\widehat{f}^{(n)}$), and let $\widetilde{\mathcal{L}}_{\widehat{\mu}_n}(z)$ be its shifted logarithmic potential.

\begin{lemma}\label{lemma:scalingTranslationInvariance}
Assume that $\widetilde{\mathcal{L}}_{\mu_n}(z) \to \Psi(z)$ in $L^1_{loc}$, where $\Psi(z)$ is the shifted logarithmic potential given by \eqref{eq:pointwiseConvergence} of the asymptotic zero-counting measure $\lim_{n\to\infty}\mu_n$. Then $\widetilde{\mathcal{L}}_{\widehat{\mu}_n}(z) \to \widehat{\Psi}(z)$ in $L^1_{loc}$, where $\widehat{\Psi}(z)$ is the shifted logarithmic potential of $\lim_{n\to\infty}\widehat{\mu}_n$.
\end{lemma}
\begin{proof}
First note that Theorem 1, and, in particular, the $L^1_{loc}$-convergence to $\Psi(z)$ in part (iii) of the theorem, are not actually dependent on the fact that the polynomial $Q(z)$ is monic. For general $Q(z)$, equation \eqref{eq:logPotentialOfTheAsymptoticRootMeasure} needs to be adjusted to
\begin{equation}\label{eq:logPotentialOfTheAsymptoticRootMeasureAdjusted}
\Psi(z) = \frac{1}{q+t-1}\left(\max_{i=1,\dotsc,q}\left\{\log{\left\vert z-z_i\right\vert^{-1}}\right\} + \log{\vert Q\vert} - \log{\left(\left\vert c_q\right\vert\left\vert d_t\right\vert t\right)}\right).
\end{equation}

We see from \eqref{eq:explicitFormula} that
\begin{equation}\label{eq:fHatDeriv001}
\widehat{f}^{(n)}(z) = \frac{\widehat{P}_n(z)}{(Q(\tau z+a))^{n+1}}e^{T(\tau z+a)},
\end{equation}
for some polynomial $\widehat{P}_n(z) := \widehat{A}_n\prod_{k=1}^{n(q+t-1)+p}(z-\widehat{\alpha}_k)$. Similarly,
\begin{equation}\label{eq:fHatDeriv002}
\begin{split}
\widehat{f}^{(n)}(z) &= (f(\tau z+a))^{(n)} = \tau^n f^{(n)}(\tau z+a) \\
&= \tau^n\left(\frac{P_n(\tau z+a)}{(Q(\tau z+a))^{n+1}}e^{T(\tau z+a)}\right).
\end{split}
\end{equation}
By comparing equations \eqref{eq:fHatDeriv001} and \eqref{eq:fHatDeriv002}, we see that
\begin{equation}\label{eq:PnHat}
\widehat{P}_n(z) = \tau^n P_n(\tau z+a).
\end{equation}
Consequently, by using the definitions of $\widehat{P}_n(z)$ and $P_n(z)$ in \eqref{eq:PnHat}, it follows that
\begin{equation*}\label{eq:AnHatDeriv002}
\widehat{A}_n\prod_{k=1}^{n(q+t-1)+p}(z-\widehat{\alpha}_k) = \tau^{n(q+t)+p} A_n\prod_{k=1}^{n(q+t-1)+p}\left(z-\frac{\alpha_k-a}{\tau}\right),
\end{equation*}
and thus,
\begin{equation}\label{eq:AnHatDeriv003}
\widehat{A}_n = \tau^{n(q+t)+p} A_n.
\end{equation}
As a result, by using \eqref{eq:PnHat} and \eqref{eq:AnHatDeriv003} in \eqref{eq:shiftedLogPotential002},
\begin{equation}\label{eq:displacedShiftedLogPotential}
\widetilde{\mathcal{L}}_{\widehat{\mu}_n}(z) = \frac{1}{n(q+t-1)+p}\left(\log{\left\vert\frac{\widehat{P}_n(z)}{\widehat{A}_n}\right\vert} - \log{n!}\right) = \widetilde{\mathcal{L}}_{\mu_n}(\tau z+a)-\log{\tau}.
\end{equation}
As a result of \eqref{eq:displacedShiftedLogPotential} and the assumption of the lemma, $\widetilde{\mathcal{L}}_{\widehat{\mu}_n}(z)\to\Psi(\tau z+a)-\log{\tau}$ in $L^1_{loc}$.

To see that $\widehat{\Psi}(z) := \Psi(\tau z+a)-\log{\tau}$ is the correct shifted logarithmic potential of $\lim_{n\to\infty}\widehat{\mu}_n$ (rather than some other $L^1_{loc}$-function), we also need to prove that it satisfies equation \eqref{eq:logPotentialOfTheAsymptoticRootMeasureAdjusted}. To do this, define $\widehat{c}_k$ and $\widehat{d}_k$ as the coefficients of $z^k$ in $Q(\tau z+a)$ and $T(\tau z+a)$, respectively. Then, by using the definition of $Q(z)$, we see that
\begin{equation*}\label{eq:cHatk}
Q(\tau z+a) = \sum_{k=0}^{q}\widehat{c}_kz^k = \sum_{k=0}^{q}c_k(\tau z+a)^k,
\end{equation*}
so $\widehat{c}_q = \tau^q$, and similarly, $\widehat{d}_t = \tau^t d_t$.

Furthermore, for each zero $z_i$ of $Q(z)$, $\widehat{z}_i := (z_i-a)/\tau$ is a zero of $Q(\tau z+a)$. Consequently, by using this bijective correspondence between $z_i$ and $\widehat{z}_i$, we get
\begin{equation}\label{eq:maxDisplacement}
\max_{i=1,\dotsc,q}\left\{\log{\left\vert z-\widehat{z}_i\right\vert^{-1}}\right\} = \max_{i=1,\dotsc,q}\left\{\log{\left\vert z+\frac{a-z_i}{\tau}\right\vert^{-1}}\right\} = \max_{i=1,\dotsc,q}\left\{\log{\left\vert \tau z+a-z_i\right\vert^{-1}}\right\}+\log{\tau}.
\end{equation}
Finally, by using \eqref{eq:maxDisplacement} in the right-hand side of \eqref{eq:logPotentialOfTheAsymptoticRootMeasureAdjusted} for $\widehat{f}(z)$, we see that
\begin{equation*}\label{eq:logPotDisplacement}
\begin{split}
& \frac{1}{q+t-1}\left(\max_{i=1,\dotsc,q}\left\{\log{\left\vert z-\widehat{z}_i\right\vert^{-1}}\right\} + \log{\vert Q(\tau z+a)\vert} - \log{\left(\vert \widehat{c}_q\vert\vert \widehat{d}_t\vert t\right)}\right) \\
&= \frac{1}{q+t-1}\left(\max_{i=1,\dotsc,q}\left\{\log{\left\vert \tau z+a-z_i\right\vert^{-1}}\right\}+\log{\tau} + \log{\vert Q(\tau z+a)\vert} - \log{\left(\tau^{q+t}\vert d_t\vert t\right)}\right) \\
&= \Psi(\tau z+a) - \log{\tau} = \widehat{\Psi}(z).\hspace{201pt}\qedhere
\end{split}
\end{equation*}
\end{proof}
Next, let $D_\rho(b)$ denote the open disk with center $b$ and of radius $\rho$. By choosing $b$ as one of the poles of $f(z)$, and by letting $\rho$ be sufficiently small, it follows from P\'olya's Shire theorem that $D_\rho(b)$ contains no zeros of $f^{(n)}(z)$ for all large enough $n$. More precisely, after scaling and translation, we may assume that the following holds due to Lemma \ref{lemma:scalingTranslationInvariance}:

\medskip
(*) The closed disk $\bar D_2(0)$ contains exactly one pole $z_i=0$ (so that $Q(0) = 0$).

\medskip
It follows from (*), by Proposition \ref{prop:logPotentialConvergence}, that there is a positive number $N$ such that $z\in \bar D_1(0)\subset V_i^o\implies P_n(z)\neq 0$, if $n\geq N$. Equivalently, if $n\geq N$ and $P_n(z)=0$, then $\vert z\vert > 1$.

Before we give bounds for the growth of the zeros of $P_n$, we define some additional notation for convenience. For $K\subset \mathbb{C}$, let
$$\vert z_{K,n}\vert := \prod_{z\in K: P_n(z)=0}\vert z\vert,$$
where zeros are taken with multiplicities; note that if there are no zeros of $P_n(z)$ in $K$, then $\vert z_{K,n}\vert = 1$. Furthermore, let $\D_\rho := D_\rho(0) = \{ z: \vert z\vert < \rho\}$, for $\rho > 0$, and set $m_n := \deg P_n = n(q+t-1)+p$.

\begin{lemma}\label{lemma:zeroBounds}
Assume (*). Then there are real numbers $C_1,\,C_2,$ and $N$ such that $C_1\le (1/m_n)\log{(\vert z_{\D_\rho^c,n}\vert/n!)}\le C_2$ for all $n\ge N$.
\end{lemma}
\begin{proof}
Since $Q(0) = 0$ by (*), it follows from the assumptions in Theorem \ref{thm:LogPotentialMeasure} that $P(0)\neq 0$, and $Q'(0) \neq 0$. Consequently, the recurrence relation \eqref{eq:recursion002} yields that
\begin{equation}\label{eq:recurrenceAtZero}
P_n(0) = -nQ'(0)P_{n-1}(0),\,\forall n\ge 1.
\end{equation}
Because $P_0(0) = P(0)$, the solution of \eqref{eq:recurrenceAtZero} is
\begin{equation}\label{eq:recurrenceSolution}
P_n(0) = n!(-Q'(0))^nP(0),\,\forall n\ge 0.
\end{equation}
Thus, it follows from \eqref{eq:recurrenceSolution} that
\begin{equation}\label{eq:recurrenceSolutionLogarithms}
\lim_{n\to\infty}\frac{1}{m_n}\log{\left\vert\frac{P_n(0)}{n!}\right\vert} = \lim_{n\to\infty}\left(\frac{n\log{\vert Q'(0)\vert}}{n(q+t-1)+p} + \frac{\log{\vert P(0)\vert}}{n(q+t-1)+p}\right) = \frac{\log{\vert Q'(0)\vert}}{q+t-1}.
\end{equation}
Next, we consider the situation in which $0 < \rho \le 1$. Since $\vert P_n(0)\vert = \vert A_n\vert\vert z_{\D_\rho,n}\vert\vert z_{\D_\rho^c,n}\vert$, we obtain the equation
\begin{equation}\label{eq:PnZeroLimitTerms}
\frac{1}{m_n}\log{\left\vert\frac{P_n(0)}{n!}\right\vert} = \frac{1}{m_n}\left(\log{\vert A_n\vert} + \log{\vert z_{\D_\rho,n}\vert} + \log{\vert z_{\D_\rho^c,n}\vert} - \log{n!}\right).
\end{equation}

Because $\lim_{n\to\infty} (1/m_n)\log{\vert A_n\vert} = (\log{\vert d_t\vert t})/(q+t-1)$ by \eqref{eq:limit001}, $(1/m_n)\log{\vert z_{\D_\rho,n}\vert} = 0$ for all $n\ge N$ due to Proposition \ref{prop:logPotentialConvergence}, and the fact that the left-hand side (and thus also the right-hand side) of equation \eqref{eq:PnZeroLimitTerms} converges due to the limit in \eqref{eq:recurrenceSolutionLogarithms}, it follows that
\begin{equation*}\label{eq:rhoOneLimitNumber}
\lim_{n\to\infty}\frac{1}{m_n}\log{\left(\frac{\vert z_{\D_\rho^c,n}\vert}{n!}\right)} = \frac{\log{\vert Q'(0)\vert}-\log{(\vert d_t\vert t})}{q+t-1} =: C.
\end{equation*}
Hence, for any fixed $\epsilon > 0$, we can choose $C_1=C-\epsilon$ and $C_2=C+\epsilon$. Consequently, there exists a number $N=N(\epsilon)$ such that the lemma follows in this case.

We proceed with the case $\rho > 1$. In this situation, we see from \eqref{eq:PnZeroLimitTerms} that
\begin{equation}\label{eq:PnZeroLimitTermsLimit}
\lim_{n\to\infty}\frac{1}{m_n}\left(\log{\vert z_{\D_\rho,n}\vert} + \log{\vert z_{\D_\rho^c,n}\vert} - \log{n!}\right) = C.
\end{equation}

Assume that $\lim_{n\to\infty} (1/m_n)\log{(\vert z_{\D_\rho^c,n}\vert/n!)} = \infty$, for some subsequence of $n$. In order for \eqref{eq:PnZeroLimitTermsLimit} to be valid, we must have that $\lim_{n\to\infty}(1/m_n)\log{\vert z_{\D_\rho,n}\vert} = -\infty$ over the same subsequence. Furthermore, note that there exists a number $N'$ such that all the zeros of $R_n$ in $\D_\rho$ are contained in the annulus $\{z:1\le\vert z\vert < \rho\}$ for all $n\ge N'$. Hence, $\rho > 1 \implies \vert z_{\D_\rho,n}\vert\ge 1\implies (1/m_n)\log{\vert z_{\D_\rho,n}\vert}\ge 0$ for all large enough $n$, resulting in a contradiction. Thus, there exists a number $C_2$ such that
\begin{equation}\label{eq:C2Bound}
\frac{1}{m_n}\log{\left(\frac{\vert z_{\D_\rho^c,n}\vert}{n!}\right)} \le C_2,
\end{equation}
for all $n\ge N'$.

Next, assume that $\lim_{n\to\infty} (1/m_n)\log{(\vert z_{\D_\rho^c,n}\vert/n!)} = -\infty$ for some subsequence of $n$. Then by \eqref{eq:PnZeroLimitTermsLimit}, it follows that $\lim_{n\to\infty}(1/m_n)\log{\vert z_{\D_\rho,n}\vert} = \infty$ over the same subsequence. Since the number of zeros of $P_n$ in $\D_\rho$ is at most $m_n = n(q+t-1)+p$ for any fixed $n$, it follows that $\vert z_{\D_\rho,n}\vert < \rho^{m_n}$, or equivalently, $(1/m_n)\log{\vert z_{\D_\rho,n}\vert} < (1/m_n)\log{(\rho^{m_n})} = \log{\rho}$, for all $n\ge 1$. This is another contradiction. Consequently, there exist numbers $C_1$ and $N''$ such that
\begin{equation}\label{eq:C1Bound}
C_1\le\frac{1}{m_n}\log{\left(\frac{\vert z_{\D_\rho^c,n}\vert}{n!}\right)},
\end{equation}
for all $n\ge N''$. Thus, by choosing $N = \max\{N',N''\}$, the lemma follows from \eqref{eq:C2Bound} and \eqref{eq:C1Bound} in this case.
\end{proof}

\subsection{$L_{loc}^1$-convergence of the logarithmic potentials}

Recall that we have previously proven (ii) of Theorem \ref{thm:LogPotentialMeasure} in section \ref{sec:VoronoiDiagrams}. Note that if we prove (iii) of the theorem, then the whole theorem follows, since parts (i) and (iv) are immediate consequences of (iii).

To proceed, fix a number $0<\epsilon<1$. Recall that $\D_\rho$ is the disk of fixed radius $\rho > 0$ centered at the origin, and let $U\subset\D_\rho$ be the set of points on $\D_\rho$ that are at least a distance $\epsilon$ away from $\Vor_{_S}^{_B}$. To prove that the convergence of $\widetilde{\mathcal{L}}_{\mu_n}(z)$ to $\Psi(z)$ is $L_{loc}^1$, we must show that, for arbitrary $\rho$,
\begin{equation*}\label{eq:LOneLoc001}
I_1 := \int_{\D_\rho}\left\vert \widetilde{\mathcal{L}}_{\mu_n}(z) - \Psi(z)\right\vert\,\mathrm{d}\lambda = O(\epsilon),
\end{equation*}
(that is, an $\epsilon$ can be chosen so that $I_1$ is arbitrarily close to $0$) where $\lambda$ is Lebesgue measure on $\mathbb{C}$. It is appropriate to split the integral $I_1$ into two integrals and deal with each one separately:
\begin{equation*}\label{eq:LOneLoc002}
I_1 = \int_{U}\left\vert \widetilde{\mathcal{L}}_{\mu_n}(z) - \Psi(z)\right\vert\,\mathrm{d}\lambda + \int_{\D_\rho\setminus U}\left\vert \widetilde{\mathcal{L}}_{\mu_n}(z) - \Psi(z)\right\vert\,\mathrm{d}\lambda =: I_2 + I_3.
\end{equation*}
Since $U$ is the union of $q$ compact subsets of $\D_\rho\setminus\Vor_{_S}^{_B}$, it follows from the uniform convergence in Proposition \ref{prop:logPotentialConvergence} that there exists a number $N$ such that $n\ge N$ implies that $\left\vert \widetilde{\mathcal{L}}_{\mu_n}(z) - \Psi(z)\right\vert\le\epsilon$ if $z\in U$. Hence
\begin{equation}\label{eq:easyIntegral}
I_2 = \int_{U}\left\vert \widetilde{\mathcal{L}}_{\mu_n}(z) - \Psi(z)\right\vert\,\mathrm{d}\lambda \le \pi\rho^2\epsilon = O(\epsilon).
\end{equation}

The integral $I_3$ is appropriately bounded by the triangle inequality:
\begin{equation*}\label{eq:LOneLoc003}
I_3 = \int_{\D_\rho\setminus U}\left\vert \widetilde{\mathcal{L}}_{\mu_n}(z) - \Psi(z)\right\vert\,\mathrm{d}\lambda \le \int_{\D_\rho\setminus U}\left\vert \widetilde{\mathcal{L}}_{\mu_n}(z)\right\vert\,\mathrm{d}\lambda + \int_{\D_\rho\setminus U}\left\vert \Psi(z)\right\vert\,\mathrm{d}\lambda =: I_5 + I_4.
\end{equation*}
If $M_1 := \max\{\Psi(z),\ z\in\D_\rho\}$, the last integral satisfies
\begin{equation}\label{eq:LOneLocLastIngeralInequality}
I_4\le M_1 \lambda(\D_\rho\setminus U)\le 2\ell\epsilon M_1,
\end{equation}
where $\ell$ denotes the length of $\Vor_{_S}^{_B}\cap \D_\rho$. Thus, $I_4 = O(\epsilon)$.

To deal with the last integral $I_5$, we write
\begin{equation*}\label{eq:shiftedLogPotentialSplit}
\widetilde{\mathcal{L}}_{\mu_n}(z) = \frac{1}{m_n}\left(\sum_{k=1}^{m_n}\log{\vert z-\alpha_k\vert}-\log{n!}\right) = {\widetilde{\mathcal{L}}^o}_{\mu_n}(z) + {\widetilde{\mathcal{L}}^i}_{\mu_n}(z),
\end{equation*}
where $${\widetilde{\mathcal{L}}^o}_{\mu_n}(z) := (1/m_n)\left(\sum_{\vert\alpha_k\vert\ge\rho+1}\log{\vert z-\alpha_k\vert} - \log{n!}\right)$$ and $${\widetilde{\mathcal{L}}^i}_{\mu_n}(z) := (1/m_n)\left(\sum_{\vert\alpha_k\vert < \rho+1}\log{\vert z-\alpha_k\vert}\right).$$ Thus, by using the triangle inequality again,
\begin{equation*}\label{eq:I5IntSplit}
I_5 \le \int_{\D_\rho\setminus U}\left\vert {\widetilde{\mathcal{L}}^o}_{\mu_n}(z)\right\vert\,\mathrm{d}\lambda + \int_{\D_\rho\setminus U}\left\vert {\widetilde{\mathcal{L}}^i}_{\mu_n}(z)\right\vert\,\mathrm{d}\lambda =: I_6 + I_7.
\end{equation*}

Consequently, for such $\rho$,
 \begin{equation*}
0\leq \log{\vert z-\alpha_k\vert}\leq \log{(\rho+\vert \alpha_k\vert)}\leq \log{(\rho+1)} + \log{\vert\alpha_k\vert},\hskip 0.5 cm \text{if } \vert z\vert < \rho,\,\vert\alpha_k\vert\geq \rho+1,
\end{equation*}
so it follows that
\begin{equation}\label{eq:integralInnerDiskConvergence}
\begin{split}
I_6 &= \int_{\D_\rho\setminus U}\left\vert {\widetilde{\mathcal{L}}^o}_{\mu_n}(z)\right\vert\,\mathrm{d}\lambda \\
&\le\frac{1}{m_n}\int_{\D_\rho\setminus U}\left\vert\sum_{\vert\alpha_k\vert\ge\rho+1}(\log{(\rho+1)}+\log{\vert\alpha_k\vert})-\log{n!}\right\vert\,\mathrm{d}\lambda \\
&\le\int_{\D_\rho\setminus U}\left\vert\log{(\rho+1)} + \frac{1}{m_n}\log{\left(\frac{\vert z_{\D_{\rho+1}^c,n}\vert}{n!}\right)}\right\vert\,\mathrm{d}\lambda \\
&\le \left(\log{(\rho+1)} + \max\{\vert C_1\vert,\vert C_2\vert\}\right)\lambda(\D_\rho\setminus U) = O(\epsilon),
\end{split}
\end{equation}
where the last inequality holds for all sufficiently large $n$ due to Lemma \ref{lemma:zeroBounds}. (Also note that the inequality $\log{(\rho+\vert \alpha_k\vert)}\leq \log{(\rho+1)} + \log{\vert\alpha_k\vert}$ corrects a minor mistake in \cite{BoHa}, where the corresponding, incorrect inequality was $\log{(\rho+\vert \alpha_k\vert)}\leq \log{\rho} + \log{\vert\alpha_k\vert}$.)

Finally, if in addition to $\vert z\vert < \rho$ and $\vert\alpha_k\vert < \rho+1$, we also have $\vert z-\alpha_k\vert > \epsilon$, then $\vert\log{\vert z-\alpha_k\vert}\vert < \max\{-\log{\epsilon},\log{(2\rho+1)}\}$. This leads to the inequalities
\begin{equation}\label{eq:requiredLogBound}
\begin{split}
&\int_{\D_\rho\setminus U}\left\vert\log{\vert z-\alpha_k\vert}\right\vert\,\mathrm{d}\lambda \\
& < \int_{\vert z-\alpha_k\vert\le\epsilon}\left\vert\log{\vert z-\alpha_k\vert}\right\vert\,\mathrm{d}\lambda + \max\{-\log{\epsilon},\log{(2\rho+1)}\}\lambda(\D_\rho\setminus U) \\
& \le 2\pi(1/2-\log{\epsilon})(\epsilon^2/2) + \max\{-\log{\epsilon},\log{(2\rho+1)}\}\epsilon = o(1).
\end{split}
\end{equation}

Consequently, from \eqref{eq:requiredLogBound},
\begin{equation}\label{eq:inequalitiesForTheLastIntegral}
\begin{split}
I_7 &= \int_{\D_\rho\setminus U}\left\vert {\widetilde{\mathcal{L}}^i}_{\mu_n}(z)\right\vert\,\mathrm{d}\lambda \\
& \le \frac{1}{m_n}\sum_{\vert\alpha_k\vert < \rho+1}\left(\int_{\D_\rho\setminus U}\left\vert\log{\vert z-\alpha_k\vert}\right\vert\,\mathrm{d}\lambda\right) = o(1),
\end{split}
\end{equation}
where the inequality in \eqref{eq:inequalitiesForTheLastIntegral} follows because the sum has at most $m_n$ terms.

As a result, part (iii) of Theorem \ref{thm:LogPotentialMeasure} (except for the statement of Proposition \ref{prop:logPotential} below, which needs to be dealt with separately) follows from the fact that the upper bounds in \eqref{eq:easyIntegral}, \eqref{eq:LOneLocLastIngeralInequality}, \eqref{eq:integralInnerDiskConvergence}, and \eqref{eq:inequalitiesForTheLastIntegral} go to $0$ when $\epsilon$ goes to $0$.

\begin{proposition}\label{prop:logPotential}
$\Psi(z)=L(z)-D$, where $L(z) := \int_{\mathbb{C}} \log\vert z-\zeta\vert \diff \mu_{_S}(\zeta)$ is the logarithmic potential of $\mu_{_S}$ and $D := (\log{(\vert d_t\vert t)})/(q+t-1)$.
\end{proposition}
\begin{proof}
We will first prove that $L(z) := \int_{\mathbb{C}} \log\vert z-\zeta\vert \diff \mu_{_S}(\zeta)$ is well-defined as a $L^1_{loc}-$function.
Let $l_{ij}= \{ z:\ \vert z-z_i\vert =\vert z-z_j\vert\}$, and use the notation of Proposition \ref{prop:notAProbabilityMeasure}. Then, for a compact set $K\subset \mathbb{C}$,
$$
\int_K\vert L(z)\vert d\lambda(z) \leq\sum_{i,j} \int_{l_{ij}} \left(\int_K\vert \log\vert z-\zeta\vert \vert\diff \lambda(z)\right)\diff\delta_{ij}(\zeta). 
$$ 
Now fix a line $l_{ij}$. An affine change of coordinates transforms $l_{ij}$ into the real axis, and then $\delta_{ij}$ is given by
$\frac{1 }{\pi}\frac{1}{1+t^2}\diff t$. Hence it suffices to prove that
$$
\int_{\mathbb{R}} \left(\int_K\frac{\vert \log\vert z-t\vert \vert}{1+t^2}\diff \lambda(z)\right)\diff t
$$
is finite. This is clear, since for large $\vert t\vert $, the integrand is approximately $\lambda(K)\log\vert t\vert /t^2$.

Secondly, we will prove that $L(z)$ has the property that 
\begin{equation}
\label{eq:Llimit}\lim_{\vert z\vert \to\infty}\left(L(z) - \frac{q-1}{q+t-1}\log\vert z\vert\right) = 0.
\end{equation}
Since $\Psi(z)$, by inspection from \eqref{eq:pointwiseConvergence}, has the property that
$$\lim_{\vert z\vert\to\infty}\left(\Psi(z) - \frac{q-1}{q+t-1}\log{\vert z\vert}\right) = \frac{\log{(\vert d_t\vert t)}}{q+t-1} = D,$$

it will follow that $\Psi(z)-L(z)$ is bounded. However, $\Psi(z)$ and $L(z)$ have by definition the same Laplacian, and hence $\Psi(z)-L(z)$ is harmonic. By Harnack's theorem, this implies that $\Psi(z)-L(z)$ is constant, and hence by taking the limit as $\vert z\vert \to\infty$, this difference is equal to $-D$.

Now to prove \eqref{eq:Llimit}\ as above, using that the total mass of $\mu_s$ is $(q-1)/(q+t-1)$, we observe that
$$\left\vert L(z)-\frac{q-1}{q+t-1}\log\vert z\vert\right\vert \leq\sum_{i,j} \int_{l_{ij}}{\left\vert \log\left\vert 1-\frac{\zeta}{z}\right\vert \right\vert}\diff\delta_{ij}(\zeta).$$
Thus, after another affine transformation, it is enough to consider
$$\int_{\mathbb{R}}\frac{\left\vert \log\left\vert 1-\frac{t}{z}\right\vert\right\vert}{1+t^2}\diff t,$$
which is easily seen to have the limit $0$ as $\vert z\vert\to\infty.$
\end{proof}

\end{document}